\documentclass[a4paper,reqno]{amsart}

\usepackage{amsmath}
\usepackage{amsthm}
\usepackage{amssymb}
\usepackage{eucal}
\usepackage{mathtools}
\usepackage{xcolor}
\usepackage{enumitem}
\usepackage{tikz}
\usepackage{float}
\usepackage[hidelinks]{hyperref}
\usepackage{cleveref}

\numberwithin{equation}{section}
\usetikzlibrary{cd,arrows,arrows.meta}
\tikzset{vertex/.style={shape=circle,draw,minimum size=3pt,inner sep=0pt}}
\tikzset{black-vertex/.style={shape=circle,draw,inner sep=3pt}}
\tikzset{red-vertex/.style={shape=circle,draw=red,dashed,inner sep=3pt}}
\tikzset{edge/.style={->,>=latex'}}
\allowdisplaybreaks
\theoremstyle{plain}

\newtheorem{theorem}{Theorem}[section]
\newtheorem{proposition}[theorem]{Proposition}
\newtheorem{lemma}[theorem]{Lemma}
\newtheorem{corollary}[theorem]{Corollary}
\theoremstyle{definition}
\newtheorem{definition}[theorem]{Definition}

\newtheorem{example}[theorem]{Example}
\theoremstyle{remark}
\newtheorem{remark}[theorem]{Remark}
\newtheorem*{remark*}{Remark}

\AddToHook{env/proposition/begin}{\crefalias{theorem}{proposition}}
\AddToHook{env/lemma/begin}{\crefalias{theorem}{lemma}}
\AddToHook{env/corollary/begin}{\crefalias{theorem}{corollary}}
\AddToHook{env/definition/begin}{\crefalias{theorem}{definition}}
\AddToHook{env/question/begin}{\crefalias{theorem}{question}}
\AddToHook{env/example/begin}{\crefalias{theorem}{example}}
\AddToHook{env/remark/begin}{\crefalias{theorem}{remark}}

\crefname{theorem}{Theorem}{Theorems}
\Crefname{theorem}{Theorem}{Theorems}
\crefname{mainthm}{Theorem}{Theorems}
\Crefname{mainthm}{Theorem}{Theorems}
\crefname{proposition}{Proposition}{Propositions}
\Crefname{proposition}{Proposition}{Propositions}
\crefname{lemma}{Lemma}{Lemmas}
\Crefname{lemma}{Lemma}{Lemmas}
\crefname{corollary}{Corollary}{Corollaries}
\Crefname{corollary}{Corollary}{Corollaries}
\crefname{definition}{Definition}{Definitions}
\Crefname{definition}{Definition}{Definitions}
\crefname{remark}{Remark}{Remarks}
\Crefname{remark}{Remark}{Remarks}

\let\cref\Cref

\newcommand{\modA}{\mathrm{mod}\,A}
\newcommand{\add}{\operatorname{add}}
\newcommand{\Gen}{\operatorname{Gen}}
\newcommand{\Cogen}{\operatorname{Cogen}}
\newcommand{\rad}{\operatorname{rad}}
\newcommand{\soc}{\operatorname{soc}}
\newcommand{\topm}{\operatorname{top}}
\newcommand{\domdim}{\operatorname{domdim}}
\newcommand{\codomdim}{\operatorname{codomdim}}
\newcommand{\gldim}{\operatorname{gldim}}
\newcommand{\pdim}{\operatorname{pdim}}
\newcommand{\idim}{\operatorname{idim}}
\newcommand{\Hom}{\operatorname{Hom}}
\newcommand{\End}{\operatorname{End}}
\newcommand{\Ext}{\operatorname{Ext}}
\newcommand{\dist}{\operatorname{dist}}
\newcommand{\length}{\ell}

\begin{document}
\raggedbottom
\title[Resolution quivers and the QF-1 property]
{Resolution quivers and the QF-1 property of higher
	Auslander Nakayama algebras}
\author{Zhe Han}
\address{School of Mathematics and Statistics, 
    Henan University, 
    Kaifeng, 
    Henan 475004, 
    P. R. China}
\email{zhehan@vip.henu.edu.cn}
\author{Dawei Shen\textsuperscript{*}} 
\thanks{\textsuperscript{*}Corresponding author.}
\address{School of Mathematics and Statistics, 
    Henan University, 
    Kaifeng, 
    Henan 475004, 
    P. R. China}
\email{sdw12345@mail.ustc.edu.cn}
\author{Jiawei Zhou}
\address{School of Mathematics and Statistics, 
    Henan University, 
    Kaifeng, 
    Henan 475004, 
    P. R. China}
\email{txh123456xc@163.com}

\renewcommand{\shortauthors}
{Zhe Han, Dawei Shen, and Jiawei Zhou}

\date{\today}

\begin{abstract}
We give a resolution-quiver-theoretic proof of the QF-1 conjecture 
of Cruz and Marczinzik for higher Auslander Nakayama algebras of 
positive even global dimension. 
The proof combines syzygy calculations in Kupisch coordinates, 
which show that the projective dimensions of quotients of a fixed 
projective module with cyclic top strictly decrease as their lengths increase, 
with a torsion-pair argument and Morita's criterion.
\end{abstract}
\maketitle

\section{Introduction}

QF-1 algebras were introduced by Thrall as a generalization
of quasi-Frobenius algebras \cite{Thrall48}.
Morita established double-centralizer criteria for QF-1
algebras \cite{Morita58}, while Ringel investigated their
structure in global dimension at most two \cite{Ringel73}.
For Nakayama algebras, Fuller obtained criteria for the
QF-1 property in terms of Kupisch series \cite{Fuller68}.
These results suggest that the QF-1 property is closely
related to the structure of projective and injective
modules.

Higher Auslander algebras arise from Iyama's
higher-dimensional Auslander--Reiten theory and its
connection with cluster-tilting modules
\cite{Iyama07AR,Iyama07Corr}.
Recall that a finite-dimensional algebra \(A\) is called
a higher Auslander algebra if
\[
\gldim A=\domdim A=g
\]
for some integer \(g\geq 1\).
For Nakayama algebras, the relationship between global
and dominant dimensions has been studied extensively.
Madsen, Marczinzik and Zaimi obtained bounds for dominant
dimensions and classification results for higher
Auslander Nakayama algebras \cite{MMZ20}.
Ringel studied concave linear Nakayama algebras which
are higher Auslander and highlighted the different
behavior of even and odd global dimensions
\cite{Ringel22}.

Recently, Cruz and Marczinzik \cite{CM} investigated
the QF-1 property of higher Auslander algebras.
They established an Auslander--Buchsbaum type formula
\[
\pdim_A X+\domdim_A X=g
\]
for every finitely generated \(A\)-module \(X\) which
is not projective-injective, where \(A\) is a higher
Auslander algebra of global dimension \(g\).
Cruz and Marczinzik proved that the condition
\[
\pdim_A\tau_g(DA)\leq g-1
\]
is necessary for a higher Auslander algebra of
global dimension \(g\) to be QF-1 \cite{CM}.
They conjectured that this condition is also
sufficient when \(g\) is positive and even.
They verified the conjecture computationally for
linear Nakayama algebras with at most \(14\) simple
modules and cyclic Nakayama algebras with at most
\(12\) simple modules.
They also constructed a counterexample to the
corresponding statement in odd global dimension.

The main result of this paper confirms their
conjecture for all higher Auslander Nakayama
algebras of positive even global dimension.
Together with the necessity established in
\cite{CM}, it gives the following characterization.

\begin{theorem}
	\label{thm:intro-main}
	Let \(A\) be a higher Auslander Nakayama algebra
	of positive even global dimension \(g\). Then \(A\) is \textup{QF-1} if and only if
	\(\pdim_A\tau_g(DA)\le g-1\).
\end{theorem}

The proof combines the torsion-theoretic approach
of Cruz and Marczinzik with a homological analysis
of resolution quivers.
Under the assumption on \(\tau_g(DA)\), their
torsion-pair result provides information about
the injective envelopes of certain simple modules.
On the other hand, Morita's criterion reduces the
QF-1 property to conditions involving the dominant
and codominant dimensions of indecomposable modules.
For Nakayama algebras, these conditions can be
expressed in terms of the composition factors of
projective non-injective modules.
The main difficulty is therefore to connect the
homological information about simple modules with
the composition series of these projective modules.

The resolution quiver provides a combinatorial tool for this connection.
For a Nakayama algebra with Kupisch series \((c_i)\),
where \(c_i=\length(P_i)\) for the indecomposable
projective module \(P_i\) with top \(S_i\), define
the integral Gustafson function by
\[
F(i)=i+c_i.
\]
Its reduction modulo the number of isomorphism classes of simple modules
determines the resolution quiver.
Resolution quivers have been studied by Ringel
\cite{Ringel13}, Shen \cite{ShenRQ,Shen17},
and Hanson and Igusa \cite{HI20}.
More recently, Shen characterized minimal
Auslander--Gorenstein Nakayama algebras in terms
of the shape and coloring of their resolution
quivers \cite{ShenMAG}.
In even self-injective dimension, this characterization
imposes strong restrictions on the predecessors
of cyclic and non-cyclic vertices.
These restrictions apply to higher Auslander
Nakayama algebras of positive even global dimension.

The key homological ingredient is the 
syzygy calculation in Kupisch coordinates.
For the uniserial module \(M[a,b)\) with top
\(S_a\) and length \(b-a\), one has
\[
\Omega^2 M[a,b)\cong M[F(a),F(b))
\]
whenever the relevant syzygies are defined.
This formula follows from known syzygy computations
for Nakayama algebras; see
\cite[Section~2.1]{Mar18} and the related
descriptions in \cite{KMM26,Ringel13}.
Our argument combines this formula with Shen's
restrictions on the resolution quiver to obtain
a strict homological monotonicity statement.

More precisely, for quotients of a fixed
indecomposable projective module with cyclic top,
we show that their projective dimensions strictly
decrease as their lengths increase.
We also derive formulas for the projective
dimensions of simple modules corresponding to
non-cyclic vertices.
These results allow us to control the composition
factors of projective non-injective modules.
Together with the torsion-pair argument, they
yield the required conditions in Morita's
criterion and hence prove the main theorem.

After this paper was completed, we learned from R.~Marczinzik that
Cruz, Marczinzik and their coauthors have independently obtained a proof 
of the QF-1 conjecture for Nakayama algebras by a different method, which uses a
parity relation between the dominant and injective dimensions of the
algebra; see \cite{CM-QF1}.

The paper is organized as follows.
Section~2 recalls Morita's criterion and the
torsion-theoretic preliminaries.
We study the integral Gustafson function
and the structure of resolution quivers in Section~3.
Section~4 establishes the homological dimension
formulas and the strict monotonicity result.
Finally, Section~5 combines these results to
prove Theorem~\ref{thm:intro-main}.

\section{Preliminaries}
\label{sec:nakayama-qf1}

Throughout the
paper, $k$ is an algebraically closed field, $A$ is a finite-dimensional
basic connected $k$-algebra, and all modules are finite-dimensional
right $A$-modules unless otherwise stated.  

Let
\[
0\longrightarrow M\longrightarrow I^0\longrightarrow I^1\longrightarrow
\cdots
\]
be a minimal injective coresolution.  The dominant dimension of \(M\) is
\[
\domdim_A M
=
\inf\{r\ge0\mid I^r\text{ is not projective}\},
\]
with value \(\infty\) if every \(I^r\) is projective.  In particular,
\(\domdim_A M\ge1\) if and only if \(I(M)\) is projective, where \(I(M)\) denotes the injective envelope of \(M\).

Dually, if
\[
\cdots\longrightarrow P_1\longrightarrow P_0\longrightarrow M\longrightarrow0
\]
is a minimal projective resolution, we define
\[
\codomdim_A M
=
\inf\{r\ge0\mid P_r\text{ is not injective}\}.
\]
Equivalently, \(\codomdim_A M=\domdim_{A^{\mathrm{op}}}D(M)\).
Thus \(\codomdim_A M\ge\nobreak1\) if and only if \(P(M)\) is injective,
where \(P(M)\) denotes the projective cover of \(M\).

A finite-dimensional algebra $A$ is called a
\emph{higher Auslander algebra} if
\[\gldim A=g=\domdim A<\infty\]
for some $g\ge1$. 
For an \(A\)-module \(M\), 
the action induces a canonical homomorphism
\[v_M\colon A^{\mathrm{op}}\longrightarrow \End_{\End_A(M)}(M).\]
The module \(M\) is called \emph{balanced} if $v_M$ is
surjective.  
A finite-dimensional algebra \(A\) is called \emph{QF-1} if every faithful \(A\)-module is balanced. For more information on QF-1
algebras, we refer to
\cite{Thrall48,Morita58,Ringel73,Tach73}; see also
\cite[Section~2.2]{CM}.

We recall the following criterion of Morita \cite{Morita58, Yamagata96}.

\begin{theorem}
	\label{thm:morita-criterion}
	Let \(A\) be a finite-dimensional algebra with \(\domdim A\ge2\). Then
	\(A\) is \textup{QF-1} if and only if every indecomposable \(A\)-module \(M\) satisfies
	\[
	\operatorname{domdim} M+\operatorname{codomdim} M\geq 1.
	\]
\end{theorem}

For an \(A\)-module \(X\), write \(\Gen X\) and \(\Cogen X\) for the
classes of modules generated and cogenerated by \(X\), respectively.  Recall
that a pair \((\mathcal T,\mathcal F)\) of full subcategories of \(\modA\) is a torsion pair if \(\Hom_A(\mathcal T,\mathcal F)=0\) and every module
\(M\) fits into a canonical exact sequence
\[
0\longrightarrow t(M)\longrightarrow M\longrightarrow M/t(M)\longrightarrow0
\]
with \(t(M)\in\mathcal T\) and \(M/t(M)\in\mathcal F\).  In particular, a
simple module belongs to exactly one of the two classes.

We will use the following result on torsion pairs; see {\cite[Proposition~4.4]{CM}}.
\begin{proposition}
	\label{prop:CM-torsion}
	Let \(A\) be a higher Auslander algebra with
global dimension \(g\ge2\), and assume that
	\(\pdim_A\tau_g(DA)\le g-1\).
	Let \(e \in A\) be an idempotent such that
	\(\add(eA)\) is the category of projective-injective 
	\(A\)-modules. Then
	\[
	\bigl(\Gen(eA),\Cogen\tau_g(DA)\bigr)
	\]
	is a torsion pair in \(\modA\), and every module in
	\(\Cogen\tau_g(DA)\) has dominant dimension at least \(1\).
\end{proposition}

\begin{corollary}
\label{cor:CM-top-control}
Let \(A\) be a higher Auslander algebra of global
dimension \(g\ge2\) satisfying
\(\pdim_A\tau_g(DA)\le g-1\).
If \(f\in A\) is a primitive idempotent such that the \(A\)-module \(fA\) is non-injective, then the \(A^\mathrm{op}\)-module \(Af\)
is injective.
\end{corollary}

\begin{proof}
Choose an idempotent \(e \in A\) such that
\(\add(eA)\) is the category of projective-injective 
\(A\)-modules.
Put \(S_f=\topm(fA)\).  We claim that
\(S_f\notin\Gen(e A)\).  Otherwise some indecomposable
projective-injective summand \(Q\) of a finite direct sum of copies of
\(e A\) admits a nonzero map \(Q\to S_f\).  Hence \(\topm Q\cong S_f\), so \(Q\) and \(fA\) are both
projective covers of \(S_f\).  Thus \(Q\cong fA\), contradicting the
non-injectivity of \(fA\).

Since \(S_f\) is simple and the pair $\bigl(\Gen(eA),\Cogen\tau_g(DA)\bigr)$ in \cref{prop:CM-torsion} is a torsion pair,
\(S_f\in\Cogen\tau_g(DA)\).
Hence \(\domdim_A S_f\ge1\), so \(I(S_f)\) is projective.  By standard
duality, \(I(S_f)=I(\topm(fA))\cong D(Af)\).
Thus \(Af\) is injective as an \(A^\mathrm{op}\)-module.
\end{proof}

\section{The integral Gustafson function}
\label{sec:Kupisch-dynamics}
Recall that a module is uniserial if its submodules are linearly ordered by inclusion. 
A finite-dimensional
algebra is called a Nakayama algebra if every
indecomposable module is uniserial.

Let $A$ be a basic connected Nakayama algebra with $n$ isomorphism classes 
of simple modules. We denote a complete set of representatives of the simple $A$-modules by $\{S_{\bar i} \mid \bar i \in \mathbb{Z}/n\mathbb{Z}\}$.
For each $i\in\mathbb Z$, let
$S_i=S_{\bar i}$, $P_i=P(S_i)$, and $c_i=\ell(P_i)$.
The composition factors of $P_i$, from top to socle, are
$S_i,S_{i+1},\ldots,S_{i+c_i-1}$.
The Kupisch series $\{c_i\}$ satisfies
$c_{i+1}\ge c_i-1$ for every $i\in\mathbb Z$.
For the general theory of Kupisch series, see \cite{ARS}.
Here, we extend the Kupisch series periodically to $\mathbb Z$.

\begin{definition}
\label{def:F}
The \emph{integral Gustafson function} associated with the
Kupisch series $\{c_i\}$ is the map
\[
F\colon\mathbb Z\longrightarrow\mathbb Z,
\qquad F(i)=i+c_i.
\]
\end{definition}

\begin{proposition}
\label{prop:F-basic}
The integral Gustafson function \(F\) has the following properties.
\begin{enumerate}[label=\textnormal{(\arabic*)}]
\item \(F\) is non-decreasing and unbounded;
\item \(F(i+n)=F(i)+n\) for every \(i\in\mathbb Z\);
\item \(F\) induces a self-map
\(
\overline F\colon\mathbb Z/n\mathbb Z\longrightarrow\mathbb Z/n\mathbb Z
\).
\end{enumerate}
\end{proposition} 

\begin{proof}
The Kupisch inequality gives
\(F(i+1)-F(i)=1+c_{i+1}-c_i\ge0\).  Also
\(F(i+n)=i+n+c_{i+n}=F(i)+n\), which makes \(F\) unbounded and allows it to
descend to \(\mathbb Z/n\mathbb Z\).
\end{proof}

Let \(\pi\colon\mathbb Z\to\mathbb Z/n\mathbb Z\) denote the quotient map.
We have the following commutative diagram:
\[
\begin{tikzcd}[column sep=4.2em,row sep=2.8em]
\mathbb Z \arrow[r,"F"] \arrow[d,"\pi"']
  & \mathbb Z \arrow[d,"\pi"] \\
\mathbb Z/n\mathbb Z \arrow[r,"\overline F"']
  & \mathbb Z/n\mathbb Z.
\end{tikzcd}
\]

Following \cite[Section~2]{ShenMAG}, let $\tau$ denote the
cyclic-order translation on the simple $A$-modules, defined by
$\tau(S_i)=S_{i+1}$, with indices taken modulo $n$.
For cyclic Nakayama algebras, this agrees with the
Auslander--Reiten translation. In the linear case, $\tau$
is understood as the translation induced by the fixed
cyclic ordering. 
The resolution quiver $R(A)$ is the functional digraph of the map
\[
\gamma(S)=\tau\soc P(S)
\]
on the set of isomorphism classes of simple $A$-modules, see \cite{Ringel13}.
Since $\soc P_i=S_{i+c_i-1}$, we have
$\gamma(S_{\bar i})=S_{\overline{i+c_i}}
=S_{\overline F(\bar i)}$,
so $R(A)$ is isomorphic to the functional digraph of $\overline F$.
We call $\overline F$ the \emph{Gustafson function} as in
\cite[Definition~1.3, Remark~1.4]{HI20}.

Following \cite[Section~3]{Ringel13}, a vertex is \emph{red} if its
projective dimension is one and \emph{black} otherwise. A vertex is
\emph{cyclic} if it lies on an oriented cycle of \(R(A)\)
\cite[Section~1]{ShenRQ}, and a \emph{leaf} if it has no predecessor.
For a vertex \(S\), let
\[
\dist(S)=\min\{r\ge0\mid\gamma^r(S)\text{ is cyclic}\},
\qquad
\dist(A)=\max_S\dist(S).
\]
For an integer representative \(a\), we abbreviate
\(d(a)=\dist(S_{\bar a})\). Thus \(d(a)\) is the distance of
\(S_{\bar a}\) to the oriented cycle of its component.

If \(C\) is an oriented cycle of \(R(A)\), define
\[
\operatorname{wt}(C)=\frac{1}{n}\sum_{S\in C}\length P(S).
\]
This is the cycle weight from \cite[Definition~1.5]{HI20}.  Its integrality
and independence of the chosen cycle are proved in
\cite[Proposition~1.1]{ShenRQ}; we denote the common
value by \(\operatorname{wt}(R(A))\).

We shall use the following facts
\cite[Lemmas~2.1, 2.5]{ShenMAG} and \cite[Lemma~3.1]{Mad05}.

\begin{lemma}
\label{lem:Shen-basic}
Let $S_i$ be a simple $A$-module. Then
\begin{enumerate}[label=\textnormal{(\arabic*)}]
\item \(P_i\) is injective if and only if \(S_{i-1}\) is black, equivalently, \(F(i)>F(i-1)\);
\item \(S_i\) is red if and only if \(c_i=c_{i+1}+1\), equivalently,
\(F(i+1)=F(i)\);
\item \(I(S_i)\) is projective if and only if \(S_{i+1}\) is not a leaf, equivalently, \(i+1\in\operatorname{Im}F\).
\end{enumerate}
\end{lemma}

Whether \(A\) has finite global dimension is determined by the
connectedness and weight of its resolution quiver
\cite[Proposition~1.1]{Shen17}.

\begin{lemma}
\label{lem:finite-gldim-weight}
The algebra \(A\) has finite global dimension if and only if its resolution
quiver is connected and has weight one.
\end{lemma}

The quotient map $\pi\colon\mathbb Z\to\mathbb Z/n\mathbb Z$
induces a surjective morphism of directed graphs
from the functional digraph of $F$ to $R(A)$. 

\begin{lemma}
	\label{cor:connected-components-lift}
	Let $\Gamma_F$ be the functional digraph of $F$,
	and let $C_1,\ldots,C_r$ be the cycles
	of $R(A)$. Then the number of connected components
	of $\Gamma_F$ is
	\[
	\sum_{j=1}^{r}\operatorname{wt}(C_j)
	=r\,\operatorname{wt}(R(A)).
	\]
\end{lemma}

\begin{proof}
	Let $C$ be a cycle of $R(A)$ with $q$ vertices
	and weight $w$. Choose an integer $a_0$ such that
	$S_{\bar a_0}\in C$, and put $a_j=F^j(a_0)$
	for $0\le j\le q$.
	Since the residue classes of $a_0,\ldots,a_{q-1}$
	are precisely the vertices of $C$, we have
	\[
	a_q-a_0
	=\sum_{j=0}^{q-1}c_{a_j}
	=\sum_{S\in C}\length P(S)=nw.
	\]
	
	Let $\widetilde C$ be the subgraph of $\Gamma_F$
	formed by all integer lifts of the vertices of $C$.
	Every vertex of $\widetilde C$ can be written
	uniquely as $a_j+kn$, where
	$0\le j<q$ and $k\in\mathbb Z$.
	Since $F(i+n)=F(i)+n$ and $a_q=a_0+nw$,
	the arrows of $\widetilde C$ are
	\[
	\begin{aligned}
		a_j+kn&\longrightarrow a_{j+1}+kn
		&& (0\le j<q-1),\\
		a_{q-1}+kn&\longrightarrow a_0+(k+w)n.
	\end{aligned}
	\]
	
	For each $0\le r<w$, these arrows give a
	bi-infinite directed path
	\[
	\begin{aligned}
		\cdots
		&\longrightarrow a_0+(r-w)n
		\longrightarrow\cdots
		\longrightarrow a_{q-1}+(r-w)n\\
		&\longrightarrow a_0+rn
		\longrightarrow\cdots
		\longrightarrow a_{q-1}+rn\\
		&\longrightarrow a_0+(r+w)n
		\longrightarrow\cdots.
	\end{aligned}
	\]
These $w$ paths are pairwise disjoint
and exhaust $\widetilde C$.
Since no arrow joins two distinct paths,
they are precisely the connected components
of $\widetilde C$.
	
Since every connected component of $R(A)$
contains a unique oriented cycle, and every
vertex eventually reaches that cycle,
the lifts of non-cyclic vertices do not change
the number of connected components.
Moreover, the lifts of distinct connected
components of $R(A)$ are disjoint, with
no arrows between them.
Hence $\Gamma_F$ has exactly
\[
\sum_{j=1}^r\operatorname{wt}(C_j)
=r\,\operatorname{wt}(R(A))
\]
connected components.
\end{proof}

\begin{corollary}
	$A$ has finite global dimension if and only if
	$\Gamma_F$ is connected.
\end{corollary}
\begin{proof}
	By \cref{lem:finite-gldim-weight,cor:connected-components-lift}.
\end{proof}

In the remainder of the paper, we assume that $A$
has finite global dimension. By \cref{lem:finite-gldim-weight},
$R(A)$ is connected and has weight one. Let \(C\) be the unique cycle of \(R(A)\), with \(q\) vertices, and put
\[
\mathcal C=\{a\in\mathbb Z\mid S_{\bar a}\text{ is a vertex of }C\}.
\]

\begin{lemma}
\label{lem:successive-lifts}
Let \(A\) be a Nakayama algebra with \(n\) simple modules,
and let \(\{x_r\}_{r\in\mathbb Z}\) be a strictly increasing enumeration of \(\mathcal C\). Then, for every \(r\in\mathbb Z\),
\[
F(x_r)=x_{r+1},\qquad
x_{r+q}=x_r+n,\qquad
F^q(x_r)=x_r+n.
\]
\end{lemma}

\begin{proof}
Fix \(r\in\mathbb Z\). The \(n\) integers in \([x_r,x_r+n)\)
represent the residue classes modulo \(n\) exactly once.
Hence \(\mathcal C\cap[x_r,x_r+n)\) has \(q\) elements.
Since \(x_r+n\in\mathcal C\), the increasing enumeration gives
\(x_{r+q}=x_r+n\).

Set \(y_t=F^t(x_r)\) for \(0\le t\le q\).
Since \(y_{t+1}-y_t=c_{y_t}\ge1\) for \(0\le t<q\), we have
\(y_0<y_1<\cdots<y_q\).
On the residue classes corresponding to \(C\), \(\overline F\)
is a cyclic permutation of length \(q\).
Since \(\overline{y_t}=\overline F^{\,t}(\overline{x_r})\),
\(\overline{y_0},\ldots,\overline{y_{q-1}}\) are distinct and represent
all vertices of \(C\).
Since $\operatorname{wt}(C)=1$, we have
\[
y_q-y_0
=\sum_{t=0}^{q-1}c_{y_t}
=\sum_{S\in C}\length P(S)
=n.
\]
Thus \(y_q=x_r+n=x_{r+q}\). The \(q\) distinct cyclic lifts
\(y_0,\ldots,y_{q-1}\) lie in \([x_r,x_r+n)\), so they exhaust
\(\mathcal C\cap[x_r,x_r+n)\). Comparing their increasing order gives
\(y_t=x_{r+t}\) for \(0\le t<q\), and the equality also holds for
\(t=q\).
\end{proof}

\begin{example}
	\label{ex:gabriel-a5-radical-square-zero}
Let $A$ be a Nakayama algebra with five simple modules
and Kupisch series $(c_0,c_1,c_2,c_3,c_4)=(2,2,2,2,1)$. Its Gabriel quiver is
	\[
	Q:\qquad
	\begin{tikzpicture}[
		baseline=(v0.base),x=1.1cm,y=1cm,
		every node/.style={text=black,inner sep=2pt},
		quiverarrow/.style={draw=black,-{Latex[length=3.7pt,width=3.2pt]},line width=0.45pt},
		relation/.style={draw=blue,line width=0.4pt,dash pattern=on 2pt off 1.5pt,line cap=round}
		]
		\node (v0) at (0,0) {$0$};
		\node (v1) at (1,0) {$1$};
		\node (v2) at (2,0) {$2$};
		\node (v3) at (3,0) {$3$};
		\node (v4) at (4,0) {$4$};
		\draw[quiverarrow] (v0.east) -- node[below,inner sep=1pt] {$\scriptstyle\alpha_0$} (v1.west);
		\draw[quiverarrow] (v1.east) -- node[below,inner sep=1pt] {$\scriptstyle\alpha_1$} (v2.west);
		\draw[quiverarrow] (v2.east) -- node[below,inner sep=1pt] {$\scriptstyle\alpha_2$} (v3.west);
		\draw[quiverarrow] (v3.east) -- node[below,inner sep=1pt] {$\scriptstyle\alpha_3$} (v4.west);
		\draw[relation] (0.54,0.02) .. controls (0.65,0.48) and (1.35,0.48) .. (1.46,0.02);
		\draw[relation] (1.54,0.02) .. controls (1.65,0.48) and (2.35,0.48) .. (2.46,0.02);
		\draw[relation] (2.54,0.02) .. controls (2.65,0.48) and (3.35,0.48) .. (3.46,0.02);
	\end{tikzpicture}.
	\]
	The dashed arcs mark the zero relations. 
	\begin{figure}[htbp]
		\begin{tikzpicture}[
			scale=0.74,
			transform shape,
			x=1cm,
			y=1cm,
			vertex/.style={inner sep=1.2pt,font=\Large},
			lab/.style={inner sep=0pt,font=\normalsize},
			arr/.style={
				-{Stealth[length=4.6pt,width=4.2pt]},
				line width=0.95pt,
				shorten <=4pt,
				shorten >=4pt
			},
			rightarr/.style={
				-{Stealth[length=5.0pt,width=4.4pt]},
				line width=0.95pt,
				shorten <=2.5pt,
				shorten >=2.5pt
			},
			cyclearr/.style={
				-{Stealth[length=4.8pt,width=4.3pt]},
				line width=0.95pt
			}
			]
			\node[vertex] (dotsL) at (-6.0,0) {$\cdots$};
			\node[vertex] (m5)    at (-4.5,0) {$-5$};
			\node[vertex] (m3)    at (-3.0,0) {$-3$};
			\node[vertex] (m1)    at (-1.5,0) {$-1$};
			\node[vertex] (z0)    at ( 0.0,0) {$0$};
			\node[vertex] (p2)    at ( 1.5,0) {$2$};
			\node[vertex] (p4)    at ( 3.0,0) {$4$};
			\node[vertex] (p5)    at ( 4.5,0) {$5$};
			\node[vertex] (dotsR) at ( 6.0,0) {$\cdots$};
			
			\draw[rightarr] (dotsL.east) -- (m5.west);
			\draw[rightarr] (m5.east)    -- (m3.west);
			\draw[rightarr] (m3.east)    -- (m1.west);
			\draw[rightarr] (m1.east)    -- (z0.west);
			\draw[rightarr] (z0.east)    -- (p2.west);
			\draw[rightarr] (p2.east)    -- (p4.west);
			\draw[rightarr] (p4.east)    -- (p5.west);
			\draw[rightarr] (p5.east)    -- (dotsR.west);
			
			\node[vertex] (mm9) at (-4.5,2.70) {$-9$};
			\node[vertex] (mm7) at (-4.5,1.35) {$-7$};
			\node[vertex] (mm4) at ( 0.0,2.70) {$-4$};
			\node[vertex] (mm2) at ( 0.0,1.35) {$-2$};
			\node[vertex] (pp1) at ( 4.5,2.70) {$1$};
			\node[vertex] (pp3) at ( 4.5,1.35) {$3$};
			
			\draw[arr] (mm9) -- (mm7);
			\draw[arr] (mm7) -- (m5);
			\draw[arr] (mm4) -- (mm2);
			\draw[arr] (mm2) -- (z0);
			\draw[arr] (pp1) -- (pp3);
			\draw[arr] (pp3) -- (p5);
			
			\node[lab,anchor=north] at ([yshift=-8pt]m5.south) {$x_{-3}$};
			\node[lab,anchor=north] at ([yshift=-8pt]m3.south) {$x_{-2}$};
			\node[lab,anchor=north] at ([yshift=-8pt]m1.south) {$x_{-1}$};
			\node[lab,anchor=north] at ([yshift=-8pt]z0.south) {$x_0$};
			\node[lab,anchor=north] at ([yshift=-8pt]p2.south) {$x_1$};
			\node[lab,anchor=north] at ([yshift=-8pt]p4.south) {$x_2$};
			\node[lab,anchor=north] at ([yshift=-8pt]p5.south) {$x_3$};
			
			\begin{scope}[shift={(0,-4.45)}]
				\node[vertex] (L2) at ( 0.0, 3.00) {$1$};
				\node[vertex] (L4) at ( 0.0, 1.90) {$3$};
				\node[vertex] (L1) at ( 0.0, 0.80) {$0$};
				\node[vertex] (L3) at ( 1.0,-0.20) {$2$};
				\node[vertex] (L5) at (-1.0,-0.20) {$4$};
				
				\draw[arr] (L2) -- (L4);
				\draw[arr] (L4) -- (L1);
				
				\draw[cyclearr]
				({cos(78)},{-0.20+sin(78)})
				arc[start angle=78,end angle=18,radius=1];
				\draw[cyclearr]
				({cos(-18)},{-0.20+sin(-18)})
				arc[start angle=-18,end angle=-162,radius=1];
				\draw[cyclearr]
				({cos(162)},{-0.20+sin(162)})
				arc[start angle=162,end angle=102,radius=1];
			\end{scope}
		\end{tikzpicture}
		\caption{The functional digraphs of $F$ and $\overline{F}$.}
		\label{fig:lift-resolution-quiver}
	\end{figure}
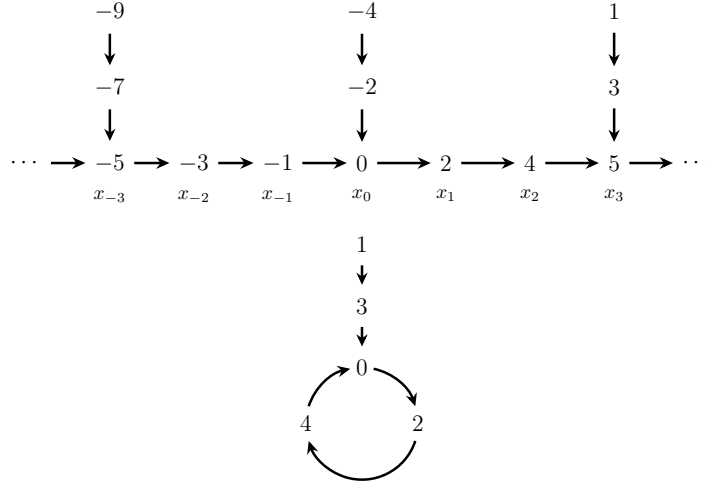
	
The resolution quiver \(R(A)\) has a unique oriented cycle
\(C\) with vertices \(0,2,4\),
	and the remaining vertices form the path
	\(1\to 3\to 0\).
	The weight of \(C\) is
	\[
	\operatorname{wt}(C)
	=\frac{c_0+c_2+c_4}{5}
	=\frac{2+2+1}{5}=1.
	\]
	
	The set of cyclic lifts is
	\[
	\mathcal C
	=\{5k,\,5k+2,\,5k+4\mid k\in\mathbb Z\}.
	\]
	Taking \(x_0=0\), their increasing enumeration is given by
	\[
	x_{3k}=5k,\qquad
	x_{3k+1}=5k+2,\qquad
	x_{3k+2}=5k+4
	\quad(k\in\mathbb Z).
	\]
	Since
	\[
	F(5k)=5k+2,\qquad
	F(5k+2)=5k+4,\qquad
	F(5k+4)=5k+5,
	\]
	we obtain
	\[
	F(x_r)=x_{r+1},\qquad
	F^3(x_r)=x_{r+3}=x_r+5
	\quad(r\in\mathbb Z).
	\]
	Notice that \(x_1=2\), whereas \(x_0+1=1\) is not a
	cyclic lift.
	
	Finally, \(S_1\) is the only leaf of \(R(A)\), and \(S_3\)
	is its only red vertex. The black vertices are
	\(S_0,S_1,S_2,S_4\). The distances to the cycle are
	\[
	(d(0),d(1),d(2),d(3),d(4))
	=(0,2,0,1,0),
	\]
	so \(\dist(A)=2\).
\end{example}

\section{Homological dimensions along the Kupisch series}
\label{sec:distance-monotonicity}
Throughout this section, let \(A\) be a higher Auslander
Nakayama algebra with
\(\gldim A=2m\), where \(m\ge1\).

A finite-dimensional algebra $A$ is called
\emph{minimal Auslander--Gorenstein} if its
self-injective dimension is finite and does not
exceed its dominant dimension
\cite[Definition~1.1]{IS}.
The following characterization of minimal
Auslander--Gorenstein Nakayama algebras of even
self-injective dimension is due to \cite[Theorem~C]{ShenMAG}.

\begin{theorem}
\label{thm:Shen-characterization}
A Nakayama algebra \(A\) is minimal Auslander--Gorenstein of
self-injective dimension \(2m\), with \(m\ge1\), if and only if the resolution quiver
satisfies the following conditions:
\begin{enumerate}[label=\textnormal{(\arabic*)}]
\item \(\dist(A)=m>0\), and every leaf \(S\) satisfies \(\dist(S)=m\);
\item every non-cyclic vertex has at most one predecessor;
\item every cyclic vertex has at most two predecessors;
\item every cyclic vertex is black.
\end{enumerate}
\end{theorem}

For a higher Auslander Nakayama algebra \(A\), we have
\[
\idim A_A=\gldim A=\domdim A=2m.
\]
Thus \(A\) is a minimal Auslander--Gorenstein algebra
of self-injective dimension \(2m\), and hence satisfies
the conditions of \cref{thm:Shen-characterization}.

By \cref{lem:finite-gldim-weight}, \(R(A)\) is connected and has
weight one. Let \(q\) be the number of vertices on its unique cycle,
and choose \(\{x_r\}\) as in \cref{lem:successive-lifts}. Then
\[F(x_r)=x_{r+1},\qquad x_{r+q}=x_r+n\]
Thus
\[	1\le x_{r+1}-x_r\le n.\]
In particular, for any integers \(x_r<a<b\le x_{r+1}\),
we have \(0<b-a<n\), so \(a\not\equiv b\pmod n\).

For integers \(a<b\le F(a)\), define
\begin{equation*}
M[a,b):=P_a/\rad^{\,b-a}P_a.
\end{equation*}
This is the uniserial module with top \(S_a\) and length \(b-a\).
Set \(M[a,a)=0\). Simultaneous translation of both endpoints by
\(n\) preserves the isomorphism class. In particular,
\(P_a=M[a,F(a))\) and \(P_{x_r}=M[x_r,x_{r+1})\).

The following lemma gives the standard syzygy formulas in interval
notation; see \cite[Section~2.1 and Proposition~2.2]{Mar18}.
It holds for arbitrary Nakayama algebras.

\begin{lemma}
\label{lem:interval-syzygies}
For \(a<b<F(a)\), there are isomorphisms
\begin{equation}
\Omega M[a,b)\cong M[b,F(a)),\qquad
\Omega^2M[a,b)\cong M[F(a),F(b)).
\label{eq:interval-syzygies}
\end{equation}
If \(b=F(a)\), then \(M[a,b)=P_a\) and its first syzygy is zero.
\end{lemma}

\begin{proof}
For \(a<b<F(a)\), the quotient map
\(P_a\twoheadrightarrow M[a,b)\) is a projective cover, since its kernel
\(\rad^{\,b-a}P_a\) is contained in \(\rad P_a\).
This kernel has top \(S_b\) and length \(F(a)-b\), so it is
\(M[b,F(a))\). Since \(F\) is non-decreasing, we have \(F(a)\le F(b)\le F(F(a))\).
The projective cover of \(M[b,F(a))\) is \(P_b\), with kernel
\(\rad^{\,F(a)-b}P_b\cong M[F(a),F(b))\);
this kernel is zero when \(F(b)=F(a)\).
Thus we obtain an exact sequence
\begin{equation}
	0\longrightarrow M[F(a),F(b))
	\longrightarrow P_b\longrightarrow P_a
	\longrightarrow M[a,b)\longrightarrow0.
	\label{eq:two-step-projective-resolution}
\end{equation}
The kernels of these projective covers give
\eqref{eq:interval-syzygies}.
If \(b=F(a)\), the kernel of the quotient map is
\(\rad^{c_a}P_a=0\).
\end{proof}

\begin{remark*}
In \cite[Section~2.2]{KMM26}, $M[a,a+t)$ is denoted by
$(a,t)$, and the first formula in \eqref{eq:interval-syzygies}
is written as $\Omega(a,t)=(a+t,c_a-t)$
for linear Nakayama algebras.
For an indecomposable module \(M=M[a,b)\) with
\(\Omega^2M\neq0\), \eqref{eq:interval-syzygies} gives
\(\topm\Omega^2M\cong\gamma(\topm M)\) and
\(\ell(\Omega^2M)=F(b)-F(a)\).
The former formula also appears in
\cite[Section~3]{Ringel13}.

\end{remark*}

For \(x_r<y\le x_{r+1}\), define
\begin{equation*}
	U_{r,y}:=M[x_r,y)
	=P_{x_r}/\rad^{\,y-x_r}P_{x_r}.
\end{equation*}
In particular, \(U_{r,x_{r+1}}=P_{x_r}\).
The following lemma gives a homological characterization
of \(d(y)\).

\begin{lemma}
	\label{lem:strip}
	Let \(x_r<y<x_{r+1}\). Then the following hold.
	\begin{enumerate}[label=\textnormal{(\arabic*)}]
		\item We have
		\begin{equation*}
			\begin{aligned}
				&x_{r+t}<F^t(y)<x_{r+t+1}
				\quad(0\le t<d(y)),\\
				&F^{d(y)}(y)=x_{r+d(y)+1}.
			\end{aligned}
		\end{equation*}
		
		\item The syzygies of \(U_{r,y}\) are given by
		\begin{equation}
			\begin{aligned}
				\Omega^{2t}U_{r,y}
				&\cong M[x_{r+t},F^t(y))
				&& (0\le t\le d(y)),\\
				\Omega^{2t+1}U_{r,y}
				&\cong M[F^t(y),x_{r+t+1})
				&& (0\le t<d(y)).
			\end{aligned}
			\label{eq:cyclic-top-resolution}
		\end{equation}
		In particular,
		\begin{equation}
			\Omega^{2d(y)}U_{r,y}\cong P_{x_{r+d(y)}},
			\qquad
			\pdim_A U_{r,y}=2d(y).
			\label{eq:quotient-projective-dimension}
		\end{equation}
	\end{enumerate}
In addition, the projective-dimension formula holds
at \(y=x_{r+1}\).
\end{lemma}
\begin{proof}
	\noindent(1)
	We proceed by induction on \(t\).
	For \(t=0\), the assertion follows from
	\(x_r<y<x_{r+1}\).
	Suppose that \(0\le t<d(y)\) and $x_{r+t}<F^t(y)<x_{r+t+1}$.
	Since \(x_{r+t}\) represents a cyclic vertex,
	it is black by \cref{thm:Shen-characterization}.
	Hence \cref{lem:Shen-basic} gives
	\(F(x_{r+t}+1)>F(x_{r+t})\).
	Since \(F\) is non-decreasing, we obtain
	\[
	x_{r+t+1}=F(x_{r+t})<F(x_{r+t}+1)
	\le F^{t+1}(y)\le F(x_{r+t+1})=x_{r+t+2}.
	\]
	If \(t+1<d(y)\), then \(F^{t+1}(y)\) is
	non-cyclic, so the last inequality is strict.
	If \(t+1=d(y)\), then \(F^{d(y)}(y)\) is cyclic.
	Since \(x_{r+d(y)+1}\) is the only cyclic lift in
	\((x_{r+d(y)},x_{r+d(y)+1}]\), we have
	\(F^{d(y)}(y)=x_{r+d(y)+1}\).
	This completes the induction.
	
	\medskip
	\noindent(2)
	For each \(0\le t<d(y)\), part (1) 
	shows that \(M[x_{r+t},F^t(y))\) is a nonzero
	proper quotient of \(P_{x_{r+t}}\).
	Since every proper nonzero quotient of an indecomposable projective 
	$A$-module is non-projective, 
	the even syzygies are non-projective
for \(0\le t<d(y)\).
	Moreover,
	$F^t(y)<x_{r+t+1}<F^{t+1}(y)$,
	so \(M[F^t(y),x_{r+t+1})\) is nonzero and
	non-projective.
	Repeated application of \cref{lem:interval-syzygies} 
	yields \eqref{eq:cyclic-top-resolution}.
	In particular,
	\[
	\Omega^{2d(y)}U_{r,y}
	\cong M[x_{r+d(y)},x_{r+d(y)+1})
	=P_{x_{r+d(y)}}.
	\]
	Since all preceding syzygies are non-projective,
	we have \(\pdim_A U_{r,y}=2d(y)\).
	For \(y=x_{r+1}\), we have
	\(U_{r,y}=P_{x_r}\) and \(d(y)=0\).
\end{proof}

\begin{lemma}
	\label{lem:last-before-cycle}
	Let \(x_r<y<x_{r+1}\).  Then the following hold.
	\begin{enumerate}[label=\textnormal{(\arabic*)}]
		\item We have
		\begin{equation*}
			F^{d(y)-1}(y)=x_{r+d(y)}-1.
		\end{equation*}
		
		\item We have
		\begin{equation*}
			\begin{aligned}
				\Omega^{2d(y)-2}U_{r,y}
				&\cong P_{x_{r+d(y)-1}}/
				\soc P_{x_{r+d(y)-1}},\\
				\Omega^{2d(y)-1}U_{r,y}
				&\cong S_{x_{r+d(y)}-1},\\
				\Omega^{2d(y)}U_{r,y}
				&\cong P_{x_{r+d(y)}}.
			\end{aligned}
		\end{equation*}
	\end{enumerate}
\end{lemma}

\begin{proof}
	\noindent(1)
	Set \(z=F^{d(y)-1}(y)\).
	By \cref{lem:strip},
	\(x_{r+d(y)-1}<z<x_{r+d(y)}\) and
	\(F(z)=x_{r+d(y)+1}=F(x_{r+d(y)})\).
	If \(z\le x_{r+d(y)}-2\), the monotonicity of
	\(F\) gives
	\[
	F(z)=F(z+1)=F(x_{r+d(y)})=x_{r+d(y)+1}.
	\]
	The integers
	\(z,z+1,x_{r+d(y)}\) represent three distinct
	vertices modulo \(n\).
	Thus the cyclic vertex represented by
	\(x_{r+d(y)+1}\) has at least three predecessors,
	contrary to \cref{thm:Shen-characterization}.
	Hence \(z=x_{r+d(y)}-1\).
	
	\medskip
	\noindent(2)
	Combining (1) with
	\cref{lem:strip}, we obtain (2).
\end{proof}

\begin{theorem}
\label{thm:strict-decrease}
If \(x_r<u<v\le x_{r+1}\), then \(d(u)>d(v)\). Moreover, we have $\pdim_A U_{r,u}>\pdim_A U_{r,v}$.
\end{theorem}

\begin{proof}
	If \(v=x_{r+1}\), then \(d(v)=0<d(u)\). 
	Moreover, \(U_{r,v}=P_{x_r}\), whereas 
	\(\pdim_A U_{r,u}=2d(u)>0\). 
	Thus both assertions hold.

	Suppose now that \(v<x_{r+1}\) and \(d(v)\ge d(u)\).
	Since \(F\) is non-decreasing, we have
	$F^{d(u)-1}(u)\le F^{d(u)-1}(v)$.
	We claim that this inequality is strict.
	If \(d(u)=1\), this follows from \(u<v\).
	Otherwise, suppose that equality holds, and let
	\(t\ge1\) be the smallest integer such that
	\(F^t(u)=F^t(v)\).
	Then \(t\le d(u)-1\).
	By \cref{lem:strip}, the distinct integers
	\(F^{t-1}(u)\) and \(F^{t-1}(v)\) lie in
	\((x_{r+t-1},x_{r+t})\).
	They represent
	distinct vertices modulo \(n\).
	Since \(t<d(u)\), their common image is
	non-cyclic, contradicting
	\cref{thm:Shen-characterization}.
	
	Hence \(F^{d(u)-1}(u)<F^{d(u)-1}(v)\).
	By \cref{lem:last-before-cycle},
	\(F^{d(u)-1}(u)=x_{r+d(u)}-1\), while
	\cref{lem:strip} gives
	\(F^{d(u)-1}(v)<x_{r+d(u)}\).
	Thus
	\[
	x_{r+d(u)}-1
	<F^{d(u)-1}(v)<x_{r+d(u)},
	\]
	which is impossible for an integer.
	Therefore \(d(u)>d(v)\), and the result follows
	from \cref{lem:strip}.
\end{proof}

\begin{corollary}
	\label{cor:gap-consequences}
	For consecutive cyclic lifts \(x_r<x_{r+1}\),
	the following hold.
	\begin{enumerate}[label=\textnormal{(\arabic*)}]
		\item If a leaf is represented in
		\((x_r,x_{r+1})\), its representative is
		\(x_r+1\).
		
		\item \(c_{x_r}=x_{r+1}-x_r\le m+1\).
		
		\item If \(c_{x_r}=m+1\), then
		\[
		d(x_r+t)=m-t+1
		\qquad (1\le t\le m+1),
		\]
		and \(S_{x_r+1}\) is a leaf.
	\end{enumerate}
\end{corollary}

\begin{proof}
	Every leaf has distance \(m\) by
	\cref{thm:Shen-characterization}.
	Suppose that a leaf is represented by an integer
	\(v\) with \(x_r+1<v<x_{r+1}\).
	Then \cref{thm:strict-decrease} gives
	\(d(x_r+1)>d(v)=m\),
	contradicting \(\dist(A)=m\).
	This proves (1).
	
By \cref{thm:strict-decrease}, we have
\[
m\ge d(x_r+1)>d(x_r+2)>\cdots>d(x_{r+1})=0.
\]
Hence \(x_{r+1}-x_r\le m+1\), proving (2).
	
Suppose that \(c_{x_r}=m+1\).
The preceding inequalities then give
\[
d(x_r+t)=m-t+1
\qquad (1\le t\le m+1).
\]
In particular, \(d(x_r+1)=m\).
If \(S_{x_r+1}\) had a predecessor,
its distance would be \(m+1\),
contradicting \(\dist(A)=m\).
Hence \(S_{x_r+1}\) is a leaf, proving (3).
\end{proof}

We now determine the projective and dominant dimensions
of non-cyclic simple modules.

\begin{proposition}
	\label{prop:noncyclic-simple-dimensions}
	If \(S_i\) is non-cyclic, then
	\[
	\pdim_A S_i=2d(i)-1,\qquad
	\domdim_A S_i=2m-2d(i)+1\ge1.
	\]
In particular, \(I(S_i)\) is projective.
\end{proposition}

\begin{proof}
	Choose an integer representative \(i\) such that
	\(x_r<i<x_{r+1}\) for some \(r\).
	We have a short exact sequence
	\begin{equation}
		0\longrightarrow S_i
		\longrightarrow U_{r,i+1}
		\longrightarrow U_{r,i}
		\longrightarrow0.
		\label{eq:adjacent-quotients}
	\end{equation}
	Let \(p=2d(i)\ge2\). By \cref{lem:strip,thm:strict-decrease},
	we have \(\pdim_A U_{r,i}=p\) and
	\(\pdim_A U_{r,i+1}\le p-2\), since both projective dimensions are even.
	Since \(\pdim_A U_{r,i+1}\le p-2\), the
    long exact Ext sequence yields
	a natural isomorphism
	\[
	\Ext_A^{p-1}(S_i,N)
	\cong \Ext_A^p(U_{r,i},N)
	\]
	for every \(A\)-module \(N\).
	Since \(\pdim_A U_{r,i}=p\), the right-hand
	side is nonzero for some \(N\).
	Moreover, \(\Ext_A^q(S_i,N)=0\) for all
	\(N\) and \(q\ge p\).
	Hence,
	\[\pdim_A S_i=p-1=2d(i)-1.\]
	
	Since \(S_i\) is non-projective, \cite[Corollary~3.9]{CM} gives
	\[
	\pdim_A S_i+\domdim_A S_i=2m.
	\]
	Consequently,
	\(\domdim_A S_i=2m-2d(i)+1\ge1\).
	Thus \(I(S_i)\) is projective.
\end{proof}

\begin{corollary}
\label{cor:cyclic-radical-control}
If \(S_i\) is cyclic, the injective envelope of every simple composition
factor of \(\rad P_i\) is projective.
\end{corollary}

\begin{proof}
Choose \(r\in\mathbb Z\) such that \(S_i=S_{x_r}\).
	Since \(P_i\cong M[x_r,x_{r+1})\), every simple
	composition factor of \(\rad P_i\) is of the form
	\(S_j\) with \(x_r<j<x_{r+1}\).
	These simple modules are non-cyclic, since
	\(x_r\) and \(x_{r+1}\) are consecutive cyclic lifts.
	The result follows from
	\cref{prop:noncyclic-simple-dimensions}.
\end{proof}

\begin{example}
	\label{ex:kupisch-344}
	Let \(A\) be a Nakayama algebra with Kupisch series
	\((3,4,4)\).
	By \cref{lem:interval-syzygies},
	\(\pdim_A S_0=4\), \(\pdim_A S_1=3\), and
	\(\pdim_A S_2=1\).
	The injective envelopes are
	\(I(S_0)\cong M[1,4)\), \(I(S_1)\cong P_1\),
	and \(I(S_2)\cong P_2\).
	The minimal injective coresolution of \(P_0\) is
	\[
	0\longrightarrow P_0\longrightarrow P_2
	\longrightarrow P_2\longrightarrow P_1
	\longrightarrow P_1\longrightarrow I(S_0)
	\longrightarrow0.
	\]
	The first four injective terms are projective,
	whereas \(I(S_0)\) is not.
	Thus 
	\[\gldim A=\domdim A=4.\]
	
\begin{figure}[htbp]
	\begin{tikzpicture}[>=stealth]
		\node (s1) at (4.1,0) {$1$};
		\node (s2) at (5.6,0) {$2$};
		\node (s0) at (7.1,0) {$0$};
		
		\draw[->] (s1) -- (s2);
		\draw[->] (s2) -- (s0);
		\draw[->] ([xshift=-4pt,yshift=6pt]s0.east)
		arc[start angle=160,end angle=-160,radius=6mm];
	\end{tikzpicture}
	\caption{The resolution quiver of $A$.}
	\label{resolution-quiver}
\end{figure}
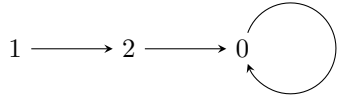

Taking \(x_0=0\), we have \(x_r=3r\) and
\((d(0),d(1),d(2))=(0,2,1)\).
By \cref{lem:strip}, the quotients
\(U_{0,t}=M[0,t)\), for \(1\le t\le3\), satisfy
\[
\pdim_A U_{0,1}=4,\qquad
\pdim_A U_{0,2}=2,\qquad
\pdim_A U_{0,3}=0.
\]
Here \(U_{0,1}=S_0\) and \(U_{0,3}=P_0\).
For the non-cyclic simple modules,
\(\domdim_A S_1=1\) and \(\domdim_A S_2=3\).
Note that \(\pdim_A S_0=2d(1)=4\), although
\(S_0\) is cyclic and \(d(0)=0\).

\end{example}

\section{Proof of the main result}
\label{sec:application}

Throughout this section, let \(A\) be a higher Auslander Nakayama algebra with global dimension \(g=2m\), where \(m\ge1\).

\begin{proposition}
\label{prop:radical-control}
Let \(P\) be an indecomposable projective non-injective 
\(A\)-module.  Then the injective envelope of every simple composition
factor of \(\rad P\) is projective.
\end{proposition}

\begin{proof}
	Choose \(i\) such that \(P\cong P_i\).
	Since \(P_i\) is not injective,
	\cref{lem:Shen-basic} implies that \(S_{i-1}\) is red.
	Hence \(c_{i-1}=c_i+1\) and
	\(F(i-1)=F(i)\).
	
	Since \(A\) has positive finite global dimension,
	it has at least two simple modules.
	Thus \(S_{i-1}\) and \(S_i\) are distinct predecessors
	of \(S_{F(i)}\).
	By \cref{thm:Shen-characterization},
	\(S_{F(i)}\) must be cyclic.
	Moreover, \(S_{i-1}\) is non-cyclic, since
	every cyclic vertex is black.
	As \(S_{F(i)}\) has at most two predecessors,
	one of which lies on the cycle,
	\(S_i\) must be cyclic.
	The result now follows from
	\cref{cor:cyclic-radical-control}.
\end{proof}

\begin{lemma}
	\label{lem:factor-hom-nakayama}
	Let \(A\) be a finite-dimensional algebra, and let
	\(e,f\) be primitive idempotents.
	Put \(S_e=eA/e\rad A\). Then \(S_e\) is a
	composition factor of \(fA\) if and only if \(fAe\ne0\).
\end{lemma}

\begin{proof}
	There is a natural isomorphism
	$\Hom_A(eA,fA)\cong fAe$.
	Assume that \(fAe\ne0\), and let
	\(\varphi:eA\to fA\) be a nonzero homomorphism.
	Since \(eA\) is an indecomposable projective module,
	it has a unique maximal submodule
	\(e\rad A=\rad(eA)\).
	The kernel of \(\varphi\) is a proper submodule
	of \(eA\), hence $\ker\varphi\subseteq e\rad A$.
	Therefore
	$\topm(\operatorname{Im}\varphi)
	\cong eA/e\rad A=S_e$.
	Thus \(S_e\) is a composition factor of
	\(\operatorname{Im}\varphi\) and of \(fA\).
	
	Conversely, suppose that \(S_e\) is a composition
	factor of \(fA\). Then there exist submodules
	\(V\subseteq U\subseteq fA\) such that
	\(U/V\cong S_e\).
	Since \(eA\) is projective, the canonical epimorphism
	\(eA\twoheadrightarrow S_e\) lifts to a nonzero
	homomorphism \(eA\to U\).
	Composing this map with the inclusion
	\(U\hookrightarrow fA\), we obtain
	$\Hom_A(eA,fA)\ne0$.
	Therefore \(fAe\ne0\).
\end{proof}

We recall the following result; see \cite{Fuller68,Tach73}.
\begin{lemma}
	\label{thm:nakayama-qf1-criterion}
	Let \(A\) be a Nakayama algebra with \(\domdim A\ge2\). Then the following
	conditions are equivalent:
	\begin{enumerate}[label=\textnormal{(\arabic*)}]
		\item \(A\) is \textup{QF-1}.
		
		\item For every indecomposable projective non-injective 
		\(A\)-module \(P\) and every simple composition factor \(S\) of \(P\),
		the injective envelope \(I(S)\) is projective.
		
		\item For any primitive idempotents \(e,f\) such that \(Ae\) and
		\(fA\) are non-injective, we have \(fAe=0\).
	\end{enumerate}
\end{lemma}

\begin{proof}
	\noindent
	\textnormal{(1)} \(\Rightarrow\) \textnormal{(2)}.
	Let \(P\) be an indecomposable projective non-injective module,
	and let \(S\) be a simple composition factor of \(P\). Choose \(r\) such that
	\(\rad^rP/\rad^{r+1}P\cong S\), and set
	\(M=P/\rad^{r+1}P\). The canonical epimorphism
	\(P\twoheadrightarrow M\) has kernel contained in \(\rad P\), hence is a
	projective cover. Thus \(P(M)\cong P\), so \(\codomdim_A M=0\).
	Moreover, \(M\) is uniserial and \(\soc M\cong S\).
	By \cref{thm:morita-criterion}, \(\domdim_A M\ge1\), and hence \(I(M)\)
	is projective. Since \(\soc M\) is essential in \(M\),
	\(I(M)\cong I(\soc M)\cong I(S)\).
	Therefore \(I(S)\) is projective.
	
	\medskip
	
	\noindent
	\textnormal{(2)} \(\Rightarrow\) \textnormal{(1)}.
	Let \(M\) be indecomposable. If \(P(M)\) is injective, then
	\(\codomdim_A M\ge1\). Otherwise, \(P(M)\) is an indecomposable
	projective non-injective module. Since \(M\) is uniserial,
	\(\soc M\) is simple and is a composition factor of \(P(M)\).
	Hence \(I(\soc M)\) is projective by \textnormal{(2)}. Since
	\(I(M)\cong I(\soc M)\), we have \(\domdim_A M\ge1\).
	Thus \cref{thm:morita-criterion} yields \textnormal{(1)}.
	
	\medskip
	
	\noindent
	\textnormal{(2)} \(\Leftrightarrow\) \textnormal{(3)}.
	Let \(e\) be a primitive idempotent. By standard duality,
	\(I(S_e)\cong D(Ae)\); hence \(I(S_e)\) is projective
	if and only if the \(A^\mathrm{op}\)-module \(Ae\) is injective. By
	\cref{lem:factor-hom-nakayama}, \(S_e\) is a composition factor of
	\(fA\) if and only if \(fAe\ne0\). Thus \textnormal{(2)} is equivalent to
	\[
	fA\text{ non-injective and }fAe\ne0
	\quad\Longrightarrow\quad
	Ae\text{ injective},
	\]
	whose contrapositive is precisely \textnormal{(3)}.
\end{proof}

Cruz and Marczinzik conjectured that a higher Auslander algebra
\(A\) of even global dimension \(g\) is \textup{QF-1} if
\(\pdim_A\tau_g(DA)\le g-1\); see
\cite[Conjecture~4.8]{CM}.
The assumption that \(g\) is even cannot be omitted:
the Nakayama algebra with Kupisch series \((3,3,3,4)\)
is a higher Auslander algebra of global dimension \(5\)
with \(\pdim_A\tau_5(DA)=3\), but is not \textup{QF-1}; see \cite[Example~4.9]{CM}. We prove this conjecture for Nakayama algebras.

\begin{theorem}
	\label{thm:main}
	\label{thm:qf1-red-nonleaf}
	Let \(A\) be a higher Auslander Nakayama algebra with
	global dimension \(g=2m\), where \(m\ge 1\).
	Then the following are equivalent:
	\begin{enumerate}[label=\textnormal{(\arabic*)}]
		\item \(A\) is \textup{QF-1};
		\item \(\pdim_A\tau_g(DA)\le g-1\);
		\item For every red vertex \(S_{i-1}\) in \(R(A)\),
		the vertex \(S_{i+1}\) is not a leaf.
	\end{enumerate}
\end{theorem}

\begin{proof}
	\((1)\Rightarrow(2)\) follows from
	\cite[Corollary 4.3]{CM}.
	
	\((2)\Rightarrow(1)\).
	Let \(P\) be an indecomposable projective non-injective
	right \(A\)-module, and choose a primitive idempotent
	\(f\in A\) such that \(P\cong fA\).
	By \cref{cor:CM-top-control}, \(Af\) is injective.
	Hence \(I(\topm P)\cong D(Af)\) is projective.
	By \cref{prop:radical-control}, the injective envelope
	of every simple composition factor of \(\rad P\)
	is also projective.
	Thus the injective envelope of every simple composition
	factor of \(P\) is projective.
	By \cref{thm:nakayama-qf1-criterion},
	\(A\) is \textup{QF-1}.
	
	\((1)\Rightarrow(3)\).
	Let \(S_{i-1}\) be red.
	By \cref{lem:Shen-basic}, \(P_i\) is not injective.
	Since \(S_i=\topm P_i\),
	\cref{thm:nakayama-qf1-criterion} implies that
	\(I(S_i)\) is projective.
	By \cref{lem:Shen-basic}, \(S_{i+1}\) is not a leaf.
	
	\((3)\Rightarrow(1)\).
	Let \(P_i\) be an indecomposable projective non-injective
	right \(A\)-module.
	By \cref{lem:Shen-basic}, \(S_{i-1}\) is red,
	so \textnormal{(3)} implies that \(S_{i+1}\) is not a leaf.
	The same lemma gives that \(I(S_i)\) is projective.
	By \cref{prop:radical-control}, the injective envelope
	of every simple composition factor of \(\rad P_i\)
	is also projective.
	Thus the injective envelope of every simple composition
	factor of \(P_i\) is projective.
	By \cref{thm:nakayama-qf1-criterion},
	\(A\) is \textup{QF-1}.
\end{proof}

\begin{remark}
	The algebra \(A\) in \cref{ex:gabriel-a5-radical-square-zero}
	satisfies \(\gldim A=\domdim A=4\).
	Since \(S_3\) is its only red vertex and \(S_0\) is not a leaf
	in \(R(A)\), \cref{thm:qf1-red-nonleaf} implies that
	\(A\) is \textup{QF-1}.
\end{remark}

\section*{Declaration of AI use}
The authors used ChatGPT (OpenAI) to explore proof strategies
and assist with developing and drafting portions of the
mathematical arguments and exposition.
The authors reviewed and revised the material incorporated
into the manuscript and take full responsibility for its
mathematical content, including the statements, proofs,
computations, and references.
\section*{Acknowledgements}
The first author is partially supported by the National Natural
Science Foundation of China (Grant Nos.~12001164 and 12571039),
and the second author is partially supported by the National Natural
Science Foundation of China (Grant Nos.~12571036 and 11801141).

\end{document}